\documentclass[11pt]{article}
\usepackage{epsfig}
\usepackage{amssymb,amsmath,amsthm,amscd}
\usepackage{latexsym}

\newtheorem{thm}{Theorem}[section]

\newtheorem{cor}[thm]{Corollary}
\newtheorem{lem}[thm]{Lemma}

\theoremstyle{definition}
\newtheorem{defn}[thm]{Definition}
\newtheorem{rem}[thm]{Remark}

\newcounter{labelflag} 

\newcommand{\Label}[1]{
                       \ifnum\thelabelflag=1
                          \ifmmode
                             \makebox[0in][l]{\qquad\fbox{\rm#1}}
                          \else
                             \marginpar{\vspace{0.7\baselineskip}
                                        \hspace{-1.1\textwidth}
                                        \fbox{\rm#1}}
                          \fi
                       \fi
                       \label{#1} }

\newcommand{\be}{\begin{equation}}
\newcommand{\ee}{\end{equation}}

\newcommand{\R}{\mathbb{R}}
\newcommand{\N}{\mathbb{N}}
\newcommand{\E}{\mathbb{E}}

\begin{document}

\baselineskip=1.3\baselineskip

\begin{titlepage}
\title{\large  \bf \baselineskip=1.3\baselineskip
Non-autonomous hybrid  stochastic  systems with delays \footnote{This work was supported
 by NSFC (11971394 and 11871049), Central Government Funds for Guiding Local Scientific and Technological Development (2021ZYD0010)
  and Fundamental Research Funds for the Central Universities (2682021ZTPY057).}  }
\vspace{10mm}

\author{{ Dingshi Li, Yusen Lin\footnote{Corresponding authors:
linyusen@my.swjtu.edu.cn (Y. Lin).}, Zhe Pu}
 \\{ \small\textsl{
 School of Mathematics, Southwest Jiaotong University, Chengdu, 610031, P. R. China}}}

\date{}
\end{titlepage}

\maketitle

{ \bf Abstract} \ \ \
The aim of this paper is to study   the dynamical behavior of
 non-autonomous  stochastic  hybrid  systems with delays. By general  Krylov-Bogolyubov's method,
we first  obtain the sufficient conditions for the   existence  of  an evolution system of measures
 of  the  non-autonomous stochastic system
 and also give some easily verifiable conditions.
 We then  prove a sufficient  condition for convergence  of evolution systems of measures
 as the delay approaches zero.  As an application
of the abstract theory, we first prove the existence  of evolution systems of measures
 for  stochastic system with time-vary delays, which
 comes from feedback
control problem based on discrete-time  state observations.
Furthermore, when observation interval  goes to zero, we show every  limit point of a   sequence of
evolution system of measures of the non-autonomous  stochastic   system
 must be a evolution system of measures of the limiting  system.

{\bf Keywords.}     Non-autonomous; Markovian switching; Delay;
Evolution system of  measures;    Limit measure.

\medskip

\section{Introduction}
\setcounter{equation}{0}

The existence   of invariant measures
  of stochastic equations was obtained   in \cite{CLW, WW1,WW2,W2019,LWW2021}.
Especially,  the limiting behavior of invariant  measures of stochastic delay  systems
 was studied in \cite{LWW2021} as delay approaches to zero.
 The existence of periodic measures  was also obtained in \cite{LWW2022,DD2008,LL2021}
for the equations with periodic time-dependent
forcings. Furthermore, in \cite{LWW2021} the authors studied
the limiting behavior of periodic measures of the stochastic equations with delays
 as delay goes to zero.
To extended the notation of periodic measures to cover the equations with aperiodic external force,
the concept of evolution system of measures was developed by \cite{DG2006,PR2007}.
 Recently, in \cite{WCT2022},  the limiting behavior of evolution system of measures
 of non-autonomous stochastic evolution  systems was studied by Wang et. al.

We mention that hybrid stochastic differential equations  (also known as
stochastic differential equations with Markovian switching) have many applications
in practice. For example, it has been used to model   systems where they may experience
abrupt changes in their structure.
For  hybrid stochastic differential equations  without  delays, the
existence  and asymptotic stability in distribution of invariant measures  was studied in
\cite{YM2003}. The asymptotic stability in distribution of invariant measures
 was also obtained in \cite{YZM2003,DDD2014,WWM2019} for
 the hybrid stochastic differential equations with constant delays.
In \cite{LLP2022}, the authors obtained the sufficient conditions of existence, stability and
 convergence of  evolution system of measures
 of non-autonomous hybrid stochastic evolution  systems.

This paper is concerned with the existence, stability  and convergence of evolution system of measures of
non-autonomous stochastic hybrid differential
equations with delays. By general  Krylov-Bogolyubov's method,
we first  obtain the sufficient conditions for the   existence  of  an evolution system of measures
 of  the  non-autonomous stochastic system
 and also give some easily verifiable conditions. For periodic Markov processes,
we show the existence of  periodic evolution systems of measures.
 We then  prove a sufficient  condition for convergence  of evolution systems of measures
 as the delay approaches zero.
As an application of our abstract results, we will investigate the existence,
asymptotic stability in distribution,
 and the limiting behavior
  of  evolution system of measures of  stochastic system with time-vary delays, which
 comes from feedback
control problem based on discrete-time  state observations. When a
given stochastic hybrid differential
equation is not stable, Mao \cite{M2013} discussed
how to design a feedback control based on discrete-time state
observations to stabilise the stochastic equation in the sense of the mean square
exponential stability. Such a stabilisation problem has since then
been studied by many authors, see, e.g.,\cite{LLLM2022,YLLMQ2015,MLHLL2014,S2017,FFMXY2019,YZ2019,YHLM2021}.

Throughout this paper,  we let $\R^n$
be the $n$-dimensional Euclidean space and $X$ be a Polish space with a metric $d$.
Let $\mathcal B(X)$ denote the family
of all Borel measurable sets in $X$. If $x\in \R^n$, then $|x|$ is
its Euclidean norm. If $A$ is a vector or matrix, its transpose is
denoted by $A^T$. Let $0<\rho\leq 1$ and $C_\rho:=C\left( {\left[ { - \rho ,0} \right],\R^n } \right)$ with the maximum norm
$\left\| \phi  \right\|_\rho   = \mathop {\sup }\limits_{ - \rho
\le s \le 0} | {\phi \left( s \right)} |,\,\,\phi  \in C_\rho.$
Let $W(t)=(W_1,\ldots,W_m)^T$ be an $m$-dimensional
standard two-side Wiener process on a complete filtered probability space $
( \Omega ,\mathcal F, \{\mathcal F_t  \}_{t\in \R} ,P )$  satisfying  the usual condition
and $r(t)$, $t\in \R$, be a right continuous irreducible Markov chain,  independent of
the Wiener process $W(t)$,  on the
probability space $
( \Omega ,\mathcal F, \{\mathcal F_t  \}_{t \in \R} ,P )$ taking values in a finite state space $S = \left\{
{1,2, \ldots ,N} \right\}$ with generator $\Gamma  = {\left(
{{r_{ij}}} \right)_{N \times N}}$ given by
\[P\left\{ {r\left( {t + \Delta } \right) = j\left| {r\left( t \right)
= i} \right.} \right\} = \left\{ {\begin{array}{*{20}{c}}
   {{r_{ij}}\Delta  + o\left( \Delta  \right),\;\;\;\ i \ne j;} \hfill  \\
   {1 + {r_{ij}}\Delta  + o\left( \Delta  \right),\;\;\;\ i = j,} \hfill  \\
\end{array}} \right.\]
where $\Delta  > 0$ and ${\lim _{\Delta  \to 0}}o\left( \Delta
\right)/\Delta  = 0, {r_{ij}} > 0$ is the transition rate from $i$
to $j$ if $i\ne j$ and ${r_{ii}} =  - \sum\nolimits_{i\ne j}
{{r_{ij}}} $. It
is well known that almost every sample path of $r(t)$ is a right-continuous step function
and $r(t)$ is ergodic.

In this paper, we study the long-term  behavior  of the following
nonautonomous stochastic  system with Markovian switching:
 \begin{align}\label{eu1}
\begin{split}
 du(t)& = f\left( {t,r\left( t \right),u\left( t \right),u\left( {t - \rho_0 \left( t \right)} \right)} \right)dt \\
 &\quad + g\left( {t,r\left( t \right),u\left( t \right),u\left( {t - \rho_0 \left( t \right)}
  \right)} \right)dW\left( t \right),\quad t > s, \\
 \end{split}
\end{align}
with initial data
\begin{equation}\label{eu2}
u \left( s+\tau\right) = \xi(\tau),\,\,-\rho\leq \tau\leq 0, \quad \text{and}\quad r(s)=j\in S,
\end{equation}
where $s\in \R$,  $u(t)\in \R^n$ is an unknown state,
$\rho_0: \R\rightarrow [0,\rho]$ is a Borel measurable  function,
 $\xi\in C_\rho$, and $f: \mathbb R\times S \times\mathbb R^n \times\mathbb R^n  \rightarrow \mathbb R^n$
 and $g: \mathbb R\times S \times\mathbb R^n \times\mathbb R^n  \rightarrow \mathbb R^{n\times m}$
 are  Borel measurable  functions satisfying the following assumption.

$(A_0)$ There is a pair of positive constants $L_f$ and $L_g$
such that for all  $t\in \R$, $j\in S$ and  $x_i,y_i\in \R^n$, $i=1,2$,
\[
\left| {f\left( {t,j,x_1 ,x_2 } \right) - f\left( {t,j,y_1 ,y_2 } \right)} \right| \le
L_f \left( {\left| {x_1  - x_2 } \right| + \left| {y_1  - y_2 } \right|} \right)
\]
and
\[
\left| {g\left( {t,j,x_1 ,x_2 } \right) - g\left( {t,j,y_1 ,y_2 } \right)} \right| \le
L_g \left( {\left| {x_1  - x_2 } \right| + \left| {y_1  - y_2 } \right|} \right),
\]

$(A_1)$ For any $j\in S$
\[
\mathop {\sup }\limits_{t \in \R} \left| {f\left( {t,j,0,0} \right)} \right| <  + \infty\quad
\text{and}\quad
\mathop {\sup }\limits_{t \in \R} \left| {g\left( {t,j,0,0} \right)} \right| <  + \infty.
\]

It is well know   (see, e.g., \cite{Mao}) that
under Condition  $(A_0)$ and $(A_1)$,
we can show  that
  for any $\xi  \in C_\rho$ and $r(s)=j\in S$,
  system \eqref{eu1}-\eqref{eu1} has a unique solution, which is
  written as   $u(t)$. To highlight the
initial values, we let $r_{s,j}(t)$ be the Markov chain starting from state $i\in S$ at $t =s $ and
denote by $u(t,s,\xi,j)$ the solution of Eq. \eqref{eu1}-\eqref{eu2}
with initial conditions $u (s,s,\xi,j) = \xi $ and
$r(s) = j$. Moreover, for  any bounded  subset $B$ of $\R^n$,
\[
\mathop {\sup }\limits_{\left( {\xi,j} \right) \in B \times S} \E\left[ {\mathop {\sup }\limits_{s \le \tau  \le t}
 | {u\left( {\tau,s,\xi ,j} \right)} |^2 } \right] < \infty\quad \quad \forall t> s.
\]

Recall that  $u_t(s,\xi,j)$
is the segment of the solution $u(t,s,\xi,j)$ given by
\[
u_t \left( {s ,\xi,j } \right)\left( \tau \right) = u\left( {t+\tau,s ,\xi,j } \right),\quad \text{for all }\quad \tau\in [-\rho,0].
\]
Notice   that
 $u_t \left( {s ,\xi,j } \right) \in L^2 \left( {\Omega ,C_\rho} \right)
$ for all $t\geq s$.

The rest of this paper is organized as follows.
Section 2  is devoted to
 the existence and periodicity
of evolution system of measures of  \eqref{eu1}-\eqref{eu2}.
In Section 3, we show the limiting behavior
of evolution system of measures of time
nonhomogeneous  Markov processes as delay goes to zero.
As an application, Section 4 is devoted to the existence , stability, periodicity  and limiting behavior
of evolution system of measures
on $C_\rho\times S$ for a controlled problem.

\section{Existence}
\setcounter{equation}{0}

 Define $H=C_\rho \times S$ with distance $
\left\| x_1-x_2 \right\|_H = \left\| \xi_1-\xi_2  \right\|_\rho   + \left| j_1-j_2 \right|,\,\,
\text{for}\,\,x_i = \left( {\xi_i ,j_i} \right) \in H$, $i=1,2$,
 and
$C_b(H)$ as the space of bounded continuous functions $\chi:H\rightarrow \R$
endowed with the norm
\[
\left\| \chi \right\|_\infty   = \mathop {\sup }\limits_{x \in H} \left| {\chi\left( x \right)} \right|,
\]
and denote by $L_b(H)$ the space of bounded Lipschitz functions on $H$. That is,
of functions $\chi\in C_b(H)$ for which
\[
\text{Lip}\left( \chi \right): = \mathop {\sup }\limits_{x_1 ,x_2  \in H} \frac{{\left| {\chi\left( {x_1 } \right)}  -
 {\chi\left( {x_2 } \right)} \right|}}{\|x_1-x_2\|_H} < \infty .
\]
The space $L_b(H)$ is endowed with the norm
\[
\left\| \chi \right\|_L  = \left\| \chi\right\|_\infty   + \text{Lip}\left( \chi \right).
\]
Let us denote by $\mathcal P(H)$ the set of probability measures on $(H,\mathcal B(H))$.
Define a metric on $\mathcal P(H)$ by
\[
\text{d}_L^* \left( {\mu _1 ,\mu _2 } \right) = \mathop {\sup }\limits_{\scriptstyle \chi \in L_b \left( H \right) \hfill \atop
  \scriptstyle \left\| \chi \right\|_L  \le 1 \hfill} \left| {\left( {\chi,\mu _1 } \right) - \left( {\chi,\mu _2 } \right)} \right|,
  \quad \mu_1,\mu_2\in \mathcal P(H).
\]

Let  $y(t,s,\xi,j)$ denote the $H$-valued process $(u_t (s,\xi,j), r_{s,j}(t))$.
 $y(t,s,\xi,j)$ be a time
nonhomogeneous  Markov process. Let $p(t,  s, \xi, j, (dy , {k}))$ denote the transition probability
of the process $y(t,s,\xi,j)$. For $A\subset \mathcal B(C_\rho)$ and $B\subset S$,
 let $P(t, s, \xi,  j, A \times B)$ denote the probability of event $\{y(t,s,\xi,j)\in A \times B\}$
given initial condition $y(s,s,\xi,j) = (\xi, j)$ at time $t=s$, i.e.,
\[
P\left( {t,s, \xi,j,A \times B} \right) =
{\int_{A\times B} {p\left( {t,s,\xi, j,(dy,k)} \right)} },
\]
where ${\int_{A\times B} {p\left( {t,s,\xi, j,(dy,k)} \right)} }=\sum\limits_{k \in B}
{\int_A {p\left( {t,s,\xi, j,(dy,k)} \right)} } $.
We define the transition evolution operator
\[
P_{s,t} \varphi \left( {\xi ,j} \right) = \E\left[ {\varphi \left( {y\left( {t,s,\xi ,j}
 \right)} \right)} \right], \quad \varphi\in C_b(H).
\]
It is easy to verify that $P_{s,t}$ is Feller, that is, $P_{s,t}:C_b(H)\rightarrow C_b(H)$, for $s<t$.
Denote by $P_{s,t}^*: \mathcal P(H)\rightarrow \mathcal P(H)$ the duality
operator of $P_{s,t}$. For $(\xi,j)\in H$,
denote by $\delta_{\xi ,j}$ the Dirac measure  concentrating  on  $(\xi,j)$.

In this section, we show existence  of   an evolution system of measures $
\left( {\mu _t } \right)_{t \in \R}$
 indexed by $\R$. An evolution system of measures $
\left( {\mu _t } \right)_{t \in \R}$ satisfies
 each $\mu _t$, $t\in \R$, is a probability measure on $H$ and
 \[
\int_{H} {P_{s,t} \varphi \left( \xi,j \right)\mu _s \left( {d\xi,j} \right)}  =
\int_{H} {\varphi \left( \xi,j \right)\mu _t \left( {d\xi,j} \right),\quad\forall
 \varphi  \in C_b \left( H \right)} ,\quad s < t.
\]
For $\varpi>0$, the evolution system of measures $\mu _t$, $t\in \R$, is $\varpi$-periodic, if
$$\mu _t=\mu _{t+\varpi},\quad  \forall t\in \R.
$$

We will apply general  Krylov-Bogolyubov's method
  to prove the existence of evolution system of  measures
  of \eqref{eu1}-\eqref{eu2}.
  To that end, fixed $(\xi,j)\in H$,
   for each $n\in \N$ and  $t\geq -n+\rho$,  define  a probability measure $\mu_{n,t}\in \mathcal P(H)$
  by
\begin{equation}\label{ie1}
\mu _{n,t}  = \frac{1}{{t -\rho+ n}}\int_{ - n}^{t-\rho} {P\left( {t,\tau,\xi ,j, \cdot  \times  \cdot } \right)} d\tau.
\end{equation}

\begin{lem}\label{Lkb}
Suppose for each $t\in \R$ the sequence $\left\{ {\mu _{n,t} } \right\}_{n = 1}^\infty$
 is tight on $H$. Then \eqref{eu1}-\eqref{eu2}
has an evolution system of measures.
\end{lem}
\begin{proof}
The proof is similar to that of Theorem 3.1 in \cite{PR2007}.
For the convenience of readers, we still prove it in detail.
By the Prokhorov theorem there exists a probability measure $\mu_t$ on $H$
 and a    subsequence (which is still denoted by
 $\left\{ {\mu _{n,t} } \right\}_{n = 1}^\infty$) such that
\begin{equation}\label{ie2}
\mu _{n,t}  \to \mu_t ,\quad \text{as}\quad n \to \infty .
\end{equation}
Let $t\in \R$ and choose $s\leq t$. Define
\[
\nu _t : = P_{s,t}^ *  \mu _s.
\]
Note that if this definition is indeed independent of $s$,
$\nu _t=\mu _t$ and $\{\mu _t\}_{t\in \R}$ is an
evolution system of measures of \eqref{eu1}-\eqref{eu2}.
By the Feller property of $P_{s,t}$ we have for
every $\varphi \in C_b(H)$
\begin{align*}
\begin{split}
 &\left( {\varphi ,\nu_t } \right) = \left( {P_{s,t} \varphi ,\mu _s } \right)
 = \mathop {\lim }\limits_{n \to \infty } \left( {P_{s,t} \varphi ,\mu _{n,s} } \right) \\
  &= \mathop {\lim }\limits_{n \to \infty } \frac{1}{{s-\rho + n}}\int_{ - n}^{s-\rho} {P_{\tau ,s}
  \left( {P_{s,t} \varphi } \right)\left( {\xi ,j} \right)} d\tau  \\
 & = \mathop {\lim }\limits_{n \to \infty } \frac{1}{{s-\rho + n}}\int_{ - n}^{s-\rho} {P_{\tau ,t} \varphi \left( {\xi ,j} \right)} d\tau  \\
  &= \mathop {\lim }\limits_{n \to \infty } \frac{1}{t-\rho+n}\left( {\int_{ - n}^{t-\rho} {P_{\tau ,t}
  \varphi \left( {\xi ,j} \right)} d\tau  - \int_{s-\rho}^{t-\rho} {P_{\tau ,t} \varphi \left( {\xi ,j} \right)} d\tau } \right),
\end{split}
\end{align*}
which is obviously independent of $s$, $s\leq t$. This completes the proof.
\end{proof}

\begin{cor}\label{Ckb}
Suppose there exists a pair of $(\xi,j)\in H$ such that for each $t\in \R$ the laws of the process
$\left\{ {u_t({s,\xi,j})} \right\}_{s+\rho\leq t}$ is tight on $
{C_\rho}$. Then \eqref{eu1}-\eqref{eu2} has an evolution system of measures.
\end{cor}
\begin{proof}
$\left\{ {u_t({s,\xi,j})} \right\}_{s+\rho\leq t}$ on $
{C_\rho}$ is tight implies that
the sequence $\left\{ {\mu _{n,t} } \right\}_{n = 1}^\infty$ defined in \eqref{ie1} is tight on $
H$. Combining  this fact and  Lemma \ref{Lkb}, the
proof is completed.

\end{proof}

\begin{cor}\label{Ckbp}
Suppose there exists a pair of $(\xi,j)\in H$ such that for each $t\in \R$,
$u_t({s,\xi,j})$ satisfies

$(\mathcal A_1)$  for any $\delta>0$, there exists a  positive constant  $R=R(\delta,\xi,j,t)$, independent of $s$, such that for $t\geq s$
\[
P\left\{ {\|{u_t\left( {s,\xi ,j} \right)}\|_\rho\leq R} \right\}>1-\delta
\]

and

$(\mathcal A_2)$   for any $\delta_1,\delta_2>0$ and $s+\rho\leq t$,
there exists $0<\eta=(\delta_1,\delta_2,\xi,j)<\rho$, independent of $s$, such that
\[
P\left\{ {\mathop {\sup }\limits_{t_2  - t_1  \le \eta ,t - \rho  \le t_1  \le t_2
\le t} \left| {u\left( {t_2,s,\xi,j } \right) - u\left( {t_1,s,\xi,j } \right)} \right| \leq\delta _1 } \right\} >1- \delta _2 .
\]

 Then \eqref{eu1}-\eqref{eu2} has an evolution system of measures.
\end{cor}
\begin{proof}
By Corollary \ref{Ckb}, it is sufficient to show
$\left\{ {u_t({s,\xi,j})} \right\}_{s+\rho\leq t}$ is tight on $
{C_\rho}$.
By $(\mathcal A_1)$, we infer that  for every $\delta>0$,
there exists $R=R(\delta,\xi,j,t)>0$ such that for $s+\rho\leq t$
\begin{equation}\label{t1}
P\left\{ {\left\| {u_t(s,\xi,j) } \right\|_\rho \leq R } \right\} >1- \frac{1}{3}\delta.
\end{equation}
By $(\mathcal A_2)$, one can verify that given $\delta>0$, for any $\delta^*>0$, there exists
 $0<\eta=\eta(\delta,\delta^*,\xi,j)<\rho$ such that for all $s\leq t-\rho$,
\begin{equation}\label{t2}
P\left\{ {\mathop {\sup }\limits_{\tau_2  - \tau_1  < \eta , - \rho  \le \tau_1 \le \tau_2
\le 0} \left| {u\left( {t+\tau_2,s,\xi,j } \right) - u\left( {t+\tau_1,s,\xi,j  }
 \right)} \right| \leq \delta^* } \right\} > 1-\frac{1}{3}\delta .
\end{equation}
Given $\delta>0$, set
\begin{align*}
 Y_{1,\delta }  &= \left\{ {v \in C_\rho:\left\| {v } \right\|_\rho \le R } \right\}, \nonumber\\
 Y_{2,\delta }  &= \{ v \in C_\rho: \text{for any}\,\, \delta^*>0,\,\, \text{there exists a}
 \,\,\eta=\eta(\delta^*)>0\,\, \text{such that }\nonumber\\
& {\mathop {\sup }\limits_{\tau_2  - \tau_1  \le \eta , - \rho  \le \tau_1 \le \tau_2
\le 0} \left| {v\left( {\tau_2 } \right) - v\left( {\tau_1  } \right)} \right| \leq \delta^*} \},\nonumber
 \end{align*}
and
\begin{equation*}\label{t10}
Y_\delta   = Y_{1,\delta }  \cap Y_{2,\delta }.
\end{equation*}
By \eqref{t1} and \eqref{t2}  we get, for all $s+\rho\leq t$,
\begin{equation*}\label{t11}
P\left( {\left\{ {u_t(s,\xi,j)  \in Y_\delta } \right\}} \right) > 1 - \delta .
\end{equation*}
By the Arzela-Ascoli
theorem $Y_\delta$ is a precompact subset of $
{C_\rho}$.
\end{proof}

\begin{lem}\label{Lkbp}
Suppose for each $t\in \R$ the sequence $\left\{ {\mu _{n,t} } \right\}_{n = 1}^\infty$
 is tight on $H$ and the  $y(t,s,\xi,j)$
are $\varpi$-periodic Markov processes. Then \eqref{eu1}-\eqref{eu2}
has a $\varpi$-periodic evolution system of measures.
\end{lem}
\begin{proof}
It follows from periodic property of $y(t,s,\xi,j)$ that
\begin{align*}
\mu _{n,t+\varpi}  &= \frac{1}{{t-\rho + n+\varpi}}\int_{ - n}^{t-\rho+\varpi} {P\left( {t+\varpi,s,\xi ,j, \cdot  \times  \cdot } \right)} ds\\
&= \frac{1}{{t-\rho + n+\varpi}}\int_{ - n-\varpi}^{t-\rho} {P\left( {t+\varpi,s+\varpi,\xi ,j, \cdot  \times  \cdot } \right)} ds\\
&= \frac{1}{{t -\rho+ n+\varpi}}\int_{ - n-\varpi}^{t-\rho} {P\left( {t,s,\xi ,j, \cdot  \times  \cdot } \right)} ds,
\end{align*}
which means that for any subsequence $\left\{ {\mu _{n_k,t} } \right\}_{k = 1}^\infty$ of
$\left\{ {\mu _{n,t} } \right\}_{n = 1}^\infty$, if $\mathop {\lim }\limits_{k \to \infty } \mu _{n_k ,t}  = \mu _t$,
then $\mathop {\lim }\limits_{k \to \infty } \mu _{n_k ,{t+\varpi}}  = \mu _{t+\varpi}=\mu_{t}$.
Following the proof process of  Lemma \ref{Lkb},  \eqref{eu1}-\eqref{eu2}
has a $\varpi$-periodic  evolution system of measures.
The proof is completed.

\end{proof}

\section{Limits    of  evolution system of measures}
\setcounter{equation}{0}

In this section, we discuss
the limiting behavior of evolution system of measures
of problem \eqref{eu1}-\eqref{eu2}
as   $\rho \to 0$
where $\rho$
is the length of delay
  in \eqref{eu1}.
We first prove an abstract theorem  to guarantee
any limiting point of evolution system of measures is still evolution system of measures.

Suppose for every $\rho\in (0,1]$,
$\psi \in  C([-\rho,0],X)$ and $j\in S$,
     $\{ Z^\rho (t, s, \psi,j), t\geq s\}$
   is  a stochastic process in the state space
   $C([-\rho,0],X)$  with initial value $Z^\rho (s, s, \psi,j)=\psi$  and $r(s)=j$  at initial time $s$.
   Similarly,  assume   for every $z\in X$ and $j\in S$,
  $\{ Z^0(t, s, z,j), t\geq s\}$  is  a
  stochastic  process
  in the  state space $X$
   with initial value $z$ at initial time $Z^0(s, s, z,j)=s$
   and $j(s)=j$. Denote $\mathcal H_\rho=(C([-\rho,0],X) \times S)$ and $\mathcal H_0=X\times S$.
    Let ${\mathcal Z}^\rho(t,s,\xi,j)$, $\rho\in [o,1]$, denote the $\mathcal H_\rho$-valued
process $(Z^\rho(t,s,\xi,j), r_{s,j}(t))$. ${\mathcal Z}^\rho(t,s,\xi,j)$, $\rho\in [o,1]$, are time
nonhomogeneous  Markov process and its probability transition operators
are Feller.

Given $\rho \in (0,1]$,
 define an operator $T_\rho:\mathcal H_\rho\rightarrow \mathcal H_0$ by
 $T_\rho (\psi,j)=(\psi(0),j)$  for $(\psi,j)\in \mathcal H_\rho$, and
   ${\mathcal T}_\rho: \mathcal H_1
 \rightarrow  \mathcal H_\rho$
 by ${\mathcal T}_\rho (\psi,j)
   =(\phi,j)$ for  $(\psi,j) \in \mathcal H_1$
 with  $\psi(s)=\phi(s)$  and  $s\in [-\rho,0]$.
 In other words,  ${\mathcal T}_\rho$ is a  restriction
 operator
 from $ \mathcal H_1$ to   $\mathcal H_\rho$.

Given  $D \subseteq C([-\rho,0],X)$,
we write   ${ T}_\rho D=
  \{ {  T}_\rho(  \psi,j) : \psi \in D,\,j\in S\}$.
Since ${T}_\rho$ is continuous, if $D$ is compact,  then so is ${ T}_\rho D$.
Similarly,   given
$D_1\subseteq C([-1,0],X)$,
we write  ${\mathcal T}_\rho D_1
=\{ {\mathcal T}_\rho( \psi,j): \psi \in D_1,\,j\in S\}$.
Note that  if $D_1$ is compact,
then so is ${\mathcal T}_\rho D_1$.

Throughout this section, we assume that
for every  compact set $K \subseteq C([-1,0],X)$, $t\geq  s$
   and $\eta>0$,
\begin{equation}\label{h}
\mathop {\lim }\limits_{\rho  \to 0} \mathop {\sup }
\limits_{(\psi,j)  \in {\mathcal T}_\rho K}
P\left(
{d \left( {Z^\rho  \left( {t, s,  \psi,j} \right) (0) , \
Z^0 \left( {t, s, \psi(0),j } \right)} \right) \ge \eta }
\right ) = 0.
\end{equation}

\begin{thm}\label{Tmc}
Assume  \eqref{h}  holds true
and   $ \rho_n  \in (0,1]  $.
Let
$\{\mu^{\rho_n}_t\}_{t\in \R}$
 be  an  evolution system of  measures
of  ${\mathcal Z}^{\rho_n}$ in  $\mathcal H_{\rho_n}$
for all $n \in \N$ and $\{\mu_t\}_{t\in \R}$ be a family of   probability
measures  on $\mathcal H_{0}$. Suppose for each $t\in \R$  $\{\mu^{\rho_n}_t  \}_{n=1}^\infty$
is tight in the sense that
for every $\epsilon>0$, there exists a compact set
$K_1 \subseteq C([-1,0],X)$ such that
\be\label{Tmpc 1}
 \mu^{\rho_n}_t ({\mathcal T}_{\rho_n} K_1)
 > 1-\epsilon\quad \text{for all } \ n\in \N.
\ee
 Then  we have:

 (i)  The sequence $\{\mu^{\rho_n}_t\circ T_{\rho_n}^{-1}\}_{n=1}^\infty$
 is tight on $\mathcal H_{0}$.

 (ii) If  $ \rho_n \to 0$ and $\mu_t$ is a probability measure
 in $\mathcal H_{0}$ such that $\mu^{\rho_n}_t\circ T_{\rho_n}^{-1}
 \rightarrow  \mu_t$ weakly,
then  $\{\mu_t\}_{t\in \R}$
must be an evolution system of  measures of ${\mathcal Z}^0$.
\end{thm}

\begin{proof}
 $(i)$.  Given $t\in \R$ and $\epsilon>0$, let    $K_1\subseteq C([-1,0],X)$
 be the compact set satisfying \eqref{Tmpc 1}.
 Denote by $K_0 = \{ \psi (0): \psi\in K_1\}$.
 Then $K_0\times S$ is a compact subset of $\mathcal H_{0}$   and
 for all $n\in \N$,
\be\label{wan0}
\mu^{\rho_n}_t\circ T_{\rho_n}^{-1}(  K_0\times S)
\geq \mu^{\rho_n}_t({\mathcal T}_{\rho_n} K_1)>1-\epsilon,
\ee
 which shows that     $\{\mu^{\rho_n}_t\circ T_{\rho_n}^{-1}\}$  is tight.

$(ii)$.  We only need to verify that for all   $\varphi\in L_b(\mathcal H_{0})$ and $s\leq t$,
\be\label{wan1}
\int_{\mathcal H_{0}} {\E \varphi\left( {\mathcal Z^0 \left( {t,s,z,j} \right)} \right)}
 \mu_s \left( {dz,j} \right) = \int_{\mathcal H_{0}} {\varphi\left(z,j \right)\mu_t \left( {dz,j} \right)}.
\ee
Notice that\begin{align*}
\int_{\mathcal H_{0}} \varphi( z,j)\mu ^{\rho_n}_t \circ T_{\rho_n}^{-1}(dz,j)
&=\int_{\mathcal H_{\rho_n}}\varphi\left( T_{\rho_n}(\psi,j )\right)\mu ^{\rho _n }_t (d\psi,j)\nonumber\\
&= \int_{\mathcal H_{\rho_n}} {\E \varphi\left( {T_{\rho_n}
 ( Z^{\rho _n }  \left( {t,s, \psi,j} \right),r_{s,j}(t))}
   \right)\mu ^{\rho _n }_s \left( {d\psi,j} \right)} \nonumber\\
   &= \int_{\mathcal H_{\rho_n}}
  {\E \varphi\left( { Z^{\rho _n } \left( {{t},s,\psi,j} \right) (0),r_{s,j}(t)}
   \right)\mu ^{\rho _n }_s \left( {d\psi,j } \right)}
\end{align*}
 which together with \eqref{wan0} yields that
\begin{align}\label{wan3}
\begin{split}
& \left| {\int_{\mathcal H_{0}} {\E g\left( {\mathcal Z^0 \left( {t,s,z,j} \right)} \right)}
 \mu^{\rho_n}_s\circ T_{\rho_n}^{-1} \left( {dz,j} \right) -
 \int_{\mathcal H_{0}} {\varphi\left( z,j \right)\mu ^{\rho_n}_t \circ T_{\rho_n}^{-1}
  \left( {dz,j} \right)} } \right| \\
 & \le \int_{\mathcal H_{\rho_n}} {\E\left| {\varphi\left( {Z^0 \left( {t,s,\psi(0),j } \right),r_{s,j}(t)} \right)
 - \varphi\left( {Z^{\rho _n } \left( {t,s,\psi,j} \right)(0),r_{s,j}(t)} \right)  } \right|}
 \mu ^{\rho _n }_s \left( {d\psi,j } \right) \\
 & \le \int_{ {\mathcal T}_{\rho_n}  K_1 }
  \E\left|
 \varphi ( {Z^0 \left( {t,s,\psi(0),j} \right ),r_{s,j}(t) } )
 - \varphi
 ( {Z^{\rho _n } \left( {t,s,\psi,j} \right)  } (0),r_{s,j}(t) )   \right|
  \mu ^{\rho _n }_s \left( {d\psi,j} \right) \\
 &+ \int_{{{\mathcal H_{\rho_n}}} \setminus {\mathcal T}_{\rho_n}
  K_1}
  \E\left|
 \varphi ( {Z^0 \left( {t,s,\psi(0),j} \right ),r_{s,j}(t) } )
 - \varphi
 ( {Z^{\rho _n } \left( {t,s,\psi,j} \right)  } (0) ,r_{s,j}(t))   \right|
  \mu ^{\rho _n }_s \left( {d\psi,j} \right)\\
   & \le \int_{ {\mathcal T}_{\rho_n}  K_1 }
  \E\left|
 \varphi ( {Z^0 \left( {t,s,\psi(0),j} \right ),r_{s,j}(t) } )
 - \varphi
 ( {Z^{\rho _n } \left( {t,s,\psi,j} \right)  } (0),r_{s,j}(t) )   \right|
  \mu ^{\rho _n }_s \left( {d\psi,j} \right) +2\epsilon\sup_{x\in {\mathcal H_{0}}} |\varphi(x)|.
  \end{split}
\end{align}
 Since $\varphi\in L_b({\mathcal H_{0}})$,  given $\epsilon>0$, there exists
  $\eta>0$ such that
  $|\varphi(y,j)-\varphi(z,j)|<\epsilon$ if   $d(y,z)<\eta$ and $j\in S$.
Then we get
\begin{align}\label{wan4}
\begin{split}
 &   \int_{ {\mathcal T}_{\rho_n}  K_1 }
  \E\left|
 \varphi ( {Z^0 \left( {t,s,\psi(0),j} \right ),r_{s,j}(t) } )
 - \varphi
 ( {Z^{\rho _n } \left( {t,s,\psi,j} \right)  } (0),r_{s,j}(t) )
  \right|
  \mu ^{\rho _n }_s \left( {d\psi,j} \right) \\
  &=  \int_{ {\mathcal T}_{\rho_n}  K_1 }
  \left (
  \int_{Y }
  \left|
  \varphi ( {Z^0 \left( {t,s,\psi(0),j} \right ),r_{s,j}(t) } )
 - \varphi
 ( {X^{\rho _n } \left( {t,s,\psi,j} \right)  } (0),r_{s,j}(t) )
   \right |   P(d\omega)
  \right )
    \mu ^{\rho _n }_s \left( {d\psi,j} \right) \\
   &\quad+
    \int_{ {\mathcal T}_{\rho_n}  K_1 }
  \left (
  \int_{Y^C }
  \left|
  \varphi ( {Z^0 \left( {t,s,\psi(0),j} \right ),r_{s,j}(t) } )
 - \varphi
 ( {Z^{\rho _n } \left( {t,s,\psi,j} \right)  } (0),r_{s,j}(t) )
   \right |   P(d\omega)
  \right )
    \mu ^{\rho _n }_s \left( {d\psi,j} \right) \\
  &\le 2\sup_{x\in \mathcal H_{0}} |\varphi(x)|
    \mathop {\sup }\limits_{(\psi,j) \in   {\mathcal T}_{\rho_n}  K_1 }
     P\left(
  {\left\{
  d \left(
{Z^{\rho _n } \left( {t,s,\psi,j} \right)  } (0), \
{Z^0 \left( {t,s,\psi(0),j} \right ) }
  \right  )
  \ge \eta  \right \} }
  \right )
  + \epsilon,
  \end{split}
 \end{align}
 where $Y=\left\{\omega\in \Omega:
  d \left(
{Z^{\rho _n } \left( {t,s,\psi,j} \right)  } (0), \
{Z^0 \left( {t,s,\psi(0),j} \right ) }
  \right  )
  \ge \eta  \right \}$ and $Y^C=\Omega-Y^C$.
 It follows from \eqref{h}  and  \eqref{wan3}-\eqref{wan4} that
\begin{align}\label{wan5}
\begin{split}
 &\mathop {\lim }\limits_{n \to \infty }
 \left| {\int_{\mathcal H_{0}} {\E g\left( {\mathcal Z^0 \left( {t,s,z,j} \right)} \right)}
 \mu^{\rho_n}_s\circ T_{\rho_n}^{-1} \left( {dz,j} \right) -
  \int_{\mathcal H_{0}} {\varphi\left( z,j \right)\mu ^{\rho_n}_t \circ T_{\rho_n}^{-1}
  \left( {dz,j} \right)} } \right|\\
 &\quad \le  \epsilon  + 2\epsilon  \sup_{x\in \mathcal H_{0}} |\varphi(x)|.
   \end{split}
\end{align}
   Since $\epsilon>0$ is arbitrary and $\mu^{\rho_n}_t\circ T_{\rho_n}^{-1}
   \rightarrow \mu_t$ weakly,
   we
 get  \eqref{wan1} from \eqref{wan5}.
 By \eqref{wan1} we know that
    $\{\mu_t\}_{t\in \R}$ is an evolution system of   measures of
 ${\mathcal Z}^ {0} $.
\end{proof}

\section{Application}
\setcounter{equation}{0}

In this section, we apply the  abstract results obtained in Section 2 and 3 to
a hybrid stochastic differential equation with delays, which comes from a control problem.

In \cite{LLLM2022}, Li et. al. investigated how to design a feedback
control based on discrete-time state observations to stabilise a
given unstable hybrid stochastic differential equation in the sense of asymptotic stability
in distribution. The specific description of the problem is as follows.
Consider an unstable hybrid stochastic differential equation
\be \label{oe}
du\left( t \right) = h\left( {r\left( t \right),u\left( t \right)} \right)dt +
\sigma \left( {r\left( t \right),u\left( t \right)} \right)dW\left( t \right), \quad t>s,
\ee
where $s\in \R$ and $u(t)\in \R^n$ is the state.
The aim  is to design
a linear feedback control $A\left( {r\left( t \right)} \right)u\left( {\left[ {t/\rho } \right]\rho } \right)
$  in the drift part so that the
controlled system
\be \label{ce}
du\left( t \right) = \left( {h\left( {r\left( t \right),u\left( t \right)} \right) + A\left( {r\left( t \right)}
\right)u\left( {\left[ {t/\rho } \right]\rho } \right)} \right)dt +
 \sigma \left( {r\left( t \right),u\left( t \right)} \right)dW\left( t \right), \quad t>s,
\ee
has an evolution system of measures $\left( {\mu _t } \right)_{t \in \R}$,
which is    asymptotically stable in distribution (defined later).
Here $A(j)\in \R^{n\times n}$, for $j\in S$,
$0<\rho\leq 1$ is a constant and $\left[ {t/\rho} \right]$
is the integer of $t/\rho$.
Define $\rho_0:\R\rightarrow [0,\rho ]$ by  for $k\in \mathbb Z$
\[
\rho _0 \left( t \right) = t - k\rho,\quad k\rho  \le t < \left( {k + 1} \right)\rho.
\]
Then Eq. \eqref{ce} can be written as
\be \label{ced}
du\left( t \right) = \left( {h\left( {r\left( t \right),u\left( t \right)} \right) + A\left( {r\left( t \right)}
\right)u\left( {t-\rho_0 \left( t \right) } \right)} \right)dt +
 \sigma \left( {r\left( t \right),u\left( t \right)} \right)dW\left( t \right), \quad t>s,
\ee

We assume
 $h: S\times \R^n \rightarrow \R^n$ and
 $\sigma: S\times \R^n \rightarrow \R^{n\times m}$
is  globally Lipschitz in the second variable  uniformly with respect to  $j\in S$.
Denote $f(r(t),u(t),u(t-\rho _0(t)))=h(r(t),u(t))+A(r(t))u(t-\rho _0(t))$
 and $g(r(t),u(t))=\sigma(r(t),u(t))$. Then the functions $f,g$ satisfy assumptions $(A_0)$ adn $( A_1)$.
It is well known that
under Condition  $(A_0)$ and $( A_1)$,
we can show  that
  for any $\xi  \in C_\rho$ and $r(s)=j\in S$,
 Eq. \eqref{ced} has a unique solution $u^\rho(t,s,\xi,j)$.
Let  $y^\rho(t,s,\xi,j)$ denote the $H$-valued process $(u_t^\rho (s,\xi,j), r_{s,j}(t))$.
Since $\rho_0(t)$ is $\rho$-periodic,  $y^\rho(t,s,\xi,j)$ is a time
$\rho$-periodic  Markov process.

We now recall the definition of
 asymptotic stability in distribution of the evolution system of measures.
\begin{defn}
The evolution system of measures $\left( {\mu _t } \right)_{t \in \R}$ of  Eq. \eqref{ced}
is said to be asymptotic stability in distribution if for any $\varphi\in C_b(H)$,
 \[
\mathop {\lim }\limits_{t \to +\infty } \left[ {P_{s,t} \varphi \left( \xi,j \right) -
\int_{H} {\varphi \left( x,j \right)\mu _t \left( {dx,j} \right)} } \right]
= 0,\quad\forall s \in \R,\quad (\xi,j) \in H.
\]
\end{defn}

In the sequence, let us assume

$(A_2)$ There exists a positive number $\beta$ and
symmetric positive definite matrices $Q_j (j\in S)$ such that
\begin{align*}
\begin{split}
 &2\left( {x - y} \right)Q_j \left[ {f\left( {j,x,x} \right) - f\left( {j,y,y} \right)} \right] \\
  &\quad+ \text{trace}\left[ {\left( {g\left( {j,x} \right) - g\left( {j,y} \right)} \right)^T Q_j \left( {g\left( {j,x} \right)
   - g\left( {j,y} \right)} \right)} \right] \\
  &\quad+ \sum\limits_{i = 1}^N {\gamma _{ji} \left( {x - y} \right)Q_i } \left( {x - y} \right)
   \le  - \beta \left| {x - y} \right|^2
\end{split}
 \end{align*}
for all $(j,x,y)\in S\times \R^n\times \R^n$.

The main idea presented in \cite{LLLM2022} is as follows:

(i) Design a feedback
control based on continuous-time  state observations
$A\left( {r\left( t \right)} \right)u\left( t \right)
$ in the drift term to stabilise the
 unstable hybrid stochastic differential equation \eqref{oe}.
A controlled system  is obtained:
 \be \label{coe}
du^0\left( t \right) = \left( {h\left( {r\left( t \right),u^0\left( t \right)} \right) + A\left( {r\left( t \right)}
\right)u^0\left( {t } \right)} \right)dt +
 \sigma \left( {r\left( t \right),u^0\left( t \right)} \right)dW\left( t \right), \quad t>s.
\ee

 (ii) Show that when the observation interval $\rho$  is sufficiently small, dynamical behaviors of the solutions
 of Eq. \eqref{ce} and Eq. \eqref{coe} have similar properties.

Repeating the scheme used in the proof Lemma 3.4, 3.5 and 3.6 in \cite{LLLM2022}, we get the following
 lemmas in turn.
\begin{lem}\label{Lcr}
Suppose $(A_0)$-$(A_1)$ hold. Then
for every  compact set $K \subseteq C([-1,0],\R^n)$, $t>s$ and $\eta>0$,
\begin{equation*}\label{u12}
\mathop {\lim }\limits_{\rho  \to 0 } \mathop {\sup }\limits_{(\xi,j)
 \in {\mathcal T}_\rho K}
 P\left(
 {|    u^\rho \left( {t,s,\xi,j } \right) -
 u^0 \left( {t,s,\xi(0),j }  \right ) |   \ge \eta }
 \right)= 0.
\end{equation*}
\end{lem}
\
\begin{lem}\label{Lcb}
Suppose $(A_0)$-$(A_2)$ hold. There exists a small enough $\rho^*>0$,
if $0<\rho\leq \rho^*$,    then  for  $s\in \R$, $(\xi,j)\in H$
any $\eta>0$, there exists a positive constant  $R=R(\eta,\xi,j)$, independent of $s$ and $\rho$, such that
for any $t\geq s$ and $0<\rho\leq \rho^*$,
\[
P\left\{ {\|{u_t^\rho\left( {s,\xi ,j} \right)}\|_\rho\leq R} \right\}>1-\delta
\]
\end{lem}

\begin{lem}\label{Lce}
Suppose $(A_0)$-$(A_2)$ hold. There exists a small enough $\rho^*>0$,
if $0<\rho\leq \rho^*$,        then for any $s\in \R$, $\eta>0$ and  bounded
 subset $B$ of $C_\rho$, there exists a $T=T(\eta,B)$, independent of $s$ and $\rho$,  such that
for $\left( {\xi _1 ,\xi _2 ,j} \right) \in B \times B \times S$ and $0<\rho\leq \rho^*$,
\[
P\left\{ {\| {(u_t^\rho\left( {s,\xi _1 ,j} \right)- u_t^\rho\left( {s,\xi _2 ,j} \right))} \|_\rho <
 \eta } \right\} \ge 1 - \eta,\quad \forall t \ge s+T.
\]

\end{lem}

By the similar argument as that of Lemma 2.3 in \cite{DDD2014}, one can easily verify that

\begin{lem}\label{Lcec}
Suppose $(A_0)$-$(A_2)$ hold.  There exists a small enough $\rho^*>0$,
if $0<\rho\leq \rho^*$,       then
 for any $(\xi,j)\in S$,    $\delta_1,\delta_2>0$ and $s+\rho\leq t$,
there exists $0<\eta=(\delta_1,\delta_2,\xi,j)<\rho$, independent of $s$, such that for $0<\rho\leq \rho^*$,
\[
P\left\{ {\mathop {\sup }\limits_{t_2  - t_1  \le \eta ,t - \rho  \le t_1  \le t_2
\le t} \left| {u^\rho\left( {t_2,s,\xi,j } \right) - u^\rho\left( {t_1,s,\xi,j }
\right)} \right| \leq\delta _1 } \right\} >1- \delta _2 .
\]

\end{lem}

\begin{thm}
Suppose $(A_0)$-$(A_1)$ hold.  There exists a small enough $\rho^*>0$,
if $0<\rho\leq \rho^*$, then    \ref{ced} has a $\rho$-periodic evolution system of measures.
\end{thm}
\begin{proof}
By Lemma \ref{Lcb} and \ref{Lcec}, we get the result  from Corollary \ref{Ckbp} and Lemma \ref{Lkbp} immediately.
\end{proof}

Moreover, the following results gives information on the asymptotic stability in
distribution of the $\rho$-periodic evolution system of measures.

\begin{thm}\label{Lfc}
Suppose $(A_0)$-$(A_2)$ hold. There exists a small enough $\rho^*>0$,
if $0<\rho\leq \rho^*$, then  Eq. \eqref{ced} has a unique $\rho$-periodic
evolution system of measures $\{\mu_t\}_{t\in \R}$, which is asymptotic stability in
distribution, i.e.,
for any  $(\xi,j)\in H$
\be\label{as}
\mathop {\lim }\limits_{t \to  + \infty } d_{\rm{L}}^ * \left( {P_{s,t}^ *
\delta_{\xi ,j} ,\mu_t } \right) = 0,\quad \forall s\in \R.
\ee
\end{thm}
\begin{proof}
By Lemma \ref{Lcb} and \ref{Lce}, we get \eqref{as}  from Theorem 2.10 in \cite{LLP2022} immediately.
\end{proof}

\begin{rem}
In \cite{LLLM2022}, the authors also investigated the asymptotic stability in
distribution of the solutions of Eq. \eqref{ced}. However,
it is only proved that the discrete time points  is asymptotic stability in
distribution, i.e.,
\be
\mathop {\lim }\limits_{n \to  + \infty } d_{\rm{L}}^ * \left( {P_{0,n \rho}^ *
\delta_{\xi ,j} ,\mu_{n\rho} } \right) = 0.
\ee
Moreover, if the intervals are not equal, the controlled system \eqref{ced}
 is not  periodic. By Theorem 2.10 in \cite{LLP2022}, we can get Eq. \eqref{ced} has a unique
evolution system of measures $\{\mu_t\}_{t\in \R}$, which is asymptotic stability in
distribution.
\end{rem}

It is well known (see, e.g., \cite{YM2003}) that under Conditions $(A_0)$-$(A_2)$,
Eq. \eqref{coe} has a unique invariant measure $\mu^0$, which is asymptotic stability in
distribution.

\begin{thm}\label{cov4}
Suppose $(A_0)$-$(A_2)$ hold and  $\rho_n \to 0$.
If $\{\mu_t^{\rho_n}\}_{t\in \R}$   is  the unique $\rho_n$-periodic
evolution system of measures of problem \eqref{ced} with $\rho$
replaced by $\rho_n$
and   $\mu^0$ is the  unique
invariant  measures of problem \eqref{coe}, then  for each $t\in \R$,
$ \mu^{\rho_{n }}_t \rightarrow \mu^{0}$ weakly.
\end{thm}

\begin{proof}

(i).
Since all uniform estimates
given in Lemma \ref{Lcb} and Lemma \ref{Lcec}
are uniform with respect to $\rho_n\in (0,1]$,
by  the arguments
 of Corollary \ref{Ckbp},
one can easily check    that for each $t\in \R$
 the set
 $\bigcup\limits_{\rho \in (0,1]}
  \mu^{\rho_n }_t $ is tight in the sense defined in Theorem \ref{Tmc}.

(ii).   By (i)
we know that $\{\mu^{\rho_n}_t\}$ is tight,
and hence by Theorem \ref{Tmc} and Lemma \ref{Lcr} we infer that
 the sequence
 $\{ \mu^{\rho_n }_t\circ T_{\rho_n}^{-1} \}_{n=1}^\infty$
 is also tight on $H$.
 Consequently,   there exists a subsequence $\rho_{n_k}$
and a probability measure $\mu^{*}_t$ such that
 $\mu^{\rho_{n_k}}_t\circ T_{\rho_{n_k} }^{-1}  \rightarrow \mu^{*}_t $
weakly.
By  Theorem \ref{Tmc} and Lemma \ref{Lcr} again, we find
  that $\{\mu^{*}_t\}_{t\in \R} $
is    the unique evolution system of measures of \eqref{coe}, which
coincides with the unique invariant measure.
\end{proof}


\begin{thebibliography}{10}

\bibitem{CLW}
Zhang Chen, Xiliang Li, and Bixiang Wang.
\newblock Invariant measures of stochastic delay lattice systems.
\newblock {\em Discrete \& Continuous Dynamical Systems-B}, 26(6):3235, 2021.

\bibitem{DD2008}
Giuseppe Da~Prato and Arnaud Debussche.
\newblock 2D stochastic Navier-Stokes equations with a time-periodic forcing
  term.
\newblock {\em Journal of Dynamics and Differential Equations}, 20(2):301--335,
  2008.

\bibitem{DG2006}
Giuseppe Da~Prato and Michael R{\"o}ckner.
\newblock Dissipative stochastic equations in Hilbert space with time dependent
  coefficients.
\newblock {\em Rendiconti Lincei-Matematica e Applicazioni}, 17(4):397--403,
  2006.

\bibitem{DDD2014}
Nguyen~Huu Du, Nguyen~Hai Dang, and Nguyen~Thanh Dieu.
\newblock On stability in distribution of stochastic differential delay
  equations with Markovian switching.
\newblock {\em Systems \& Control Letters}, 65:43--49, 2014.

\bibitem{FFMXY2019}
Chen Fei, Weiyin Fei, Xuerong Mao, Dengfeng Xia, and Litan Yan.
\newblock Stabilization of highly nonlinear hybrid systems by feedback control
  based on discrete-time state observations.
\newblock {\em IEEE Transactions on Automatic Control}, 65(7):2899--2912, 2019.

\bibitem{LLP2022}
Dingshi Li, Yusen Lin, and Zhe Pu.
\newblock Non-autonomous stochastic lattice systems with Markovian switching.
\newblock {\em arXiv preprint arXiv:2204.00776}, 2022.

\bibitem{LWW2022}
Dingshi Li, Bixiang Wang, and Xiaohu Wang.
\newblock Limiting behavior of invariant measures of stochastic delay lattice
  systems.
\newblock {\em Journal of Dynamics and Differential Equations}, DOI: 10.1007/s10884-021-10011-7.

\bibitem{LWW2021}
Dingshi Li, Bixiang Wang, and Xiaohu Wang.
\newblock Periodic measures of stochastic delay lattice systems.
\newblock {\em Journal of Differential Equations}, 272:74--104, 2021.

\bibitem{LLLM2022}
Xiaoyue Li, Wei Liu, Qi~Luo, and Xuerong Mao.
\newblock Stabilisation in distribution of hybrid stochastic differential
  equations by feedback control based on discrete-time state observations.
\newblock {\em Automatica}, page 110210, 2022.

\bibitem{LL2021}
Rongchang Liu and Kening Lu.
\newblock Statistical properties of 2D stochastic navier-stokes equations with
  time-periodic forcing and degenerate stochastic forcing.
\newblock {\em arXiv preprint arXiv:2105.00598}, 2021.

\bibitem{M2013}
Xuerong Mao.
\newblock Stabilization of continuous-time hybrid stochastic differential
  equations by discrete-time feedback control.
\newblock {\em Automatica}, 49(12):3677--3681, 2013.

\bibitem{MLHLL2014}
Xuerong Mao, Wei Liu, Liangjian Hu, Qi~Luo, and Jianqiu Lu.
\newblock Stabilization of hybrid stochastic differential equations by feedback
  control based on discrete-time state observations.
\newblock {\em Systems \& Control Letters}, 73:88--95, 2014.

\bibitem{Mao}
Xuerong Mao and Chenggui Yuan.
\newblock {\em Stochastic Differential Equations with Markovian Switching}.
\newblock 2006,  {Imperial College Press}.

\bibitem{PR2007}
Giuseppe~Da Prato and Michael R{\"o}ckner.
\newblock A note on evolution systems of measures for time-dependent stochastic
  differential equations.
\newblock  {\em Seminar on Stochastic Analysis, Random Fields and
  Applications}, pages 115--122. Springer, 2007.

\bibitem{S2017}
Jinghai Shao.
\newblock Stabilization of regime-switching processes by feedback control based
  on discrete time observations.
\newblock {\em SIAM Journal on Control and Optimization}, 55(2):724--740, 2017.

\bibitem{W2019}
Bixiang Wang.
\newblock Dynamics of stochastic reaction--diffusion lattice systems driven by
  nonlinear noise.
\newblock {\em Journal of Mathematical Analysis and Applications},
  477(1):104--132, 2019.

\bibitem{WW2}
Bixiang Wang and Renhai Wang.
\newblock Asymptotic behavior of stochastic schr{\"o}dinger lattice systems
  driven by nonlinear noise.
\newblock {\em Stochastic Analysis and Applications}, 38(2):213--237, 2020.

\bibitem{WCT2022}
Renhai Wang, Tomas Caraballo, and Nguyen~Huy Tuan.
\newblock Asymptotic stability of evolution systems of probability measures for
  nonautonomous stochastic systems: Theoretical results and applications.
\newblock {\em arXiv preprint arXiv:2203.13039}, 2022.

\bibitem{WW1}
Renhai Wang and Bixiang Wang.
\newblock Random dynamics of $p$-Laplacian lattice systems driven by
  infinite-dimensional nonlinear noise.
\newblock {\em Stochastic Processes and their Applications},
  130(12):7431--7462, 2020.

\bibitem{WWM2019}
Ya~Wang, Fuke Wu, and Xuerong Mao.
\newblock Stability in distribution of stochastic functional differential
  equations.
\newblock {\em Systems \& Control Letters}, 132:104513, 2019.

\bibitem{YZ2019}
Xuetao Yang and Quanxin Zhu.
\newblock Stabilization of stochastic retarded systems based on sampled-data
  feedback control.
\newblock {\em IEEE Transactions on Systems, Man, and Cybernetics: Systems},
  51(9):5895--5904, 2019.

\bibitem{YHLM2021}
Surong You, Liangjian Hu, Jianqiu Lu, and Xuerong Mao.
\newblock Stabilisation in distribution by delay feedback control for hybrid
  stochastic differential equations.
\newblock {\em IEEE Transactions on Automatic Control}, 2021.

\bibitem{YLLMQ2015}
Surong You, Wei Liu, Jianqiu Lu, Xuerong Mao, and Qinwei Qiu.
\newblock Stabilization of hybrid systems by feedback control based on
  discrete-time state observations.
\newblock {\em SIAM Journal on Control and Optimization}, 53(2):905--925, 2015.

\bibitem{YM2003}
Chenggui Yuan and Xuerong Mao.
\newblock Asymptotic stability in distribution of stochastic differential
  equations with Markovian switching.
\newblock {\em Stochastic Processes and their Applications}, 103(2):277--291,
  2003.

\bibitem{YZM2003}
Chenggui Yuan, Jiezhong Zou, and Xuerong Mao.
\newblock Stability in distribution of stochastic differential delay equations
  with Markovian switching.
\newblock {\em Systems \& control letters}, 50(3):195--207, 2003.

\end{thebibliography}
\end{document}